\documentclass[11pt]{amsart}
\usepackage{CJK} 
\usepackage{manfnt} 
\usepackage{hyperref}
\usepackage{amsfonts}
\usepackage{amsmath}
\usepackage{amssymb}
\usepackage{amscd}
\usepackage{graphicx}
\usepackage{enumitem}
\usepackage{latexsym}
\usepackage{color}
\usepackage{comment}
\usepackage{epsfig}
\usepackage{amsopn}
\usepackage{tikz-cd}
\usepackage{mathrsfs}
\usepackage{array}
\usepackage[usenames,dvipsnames]{pstricks}
\usepackage{epsfig}
\usepackage{pst-grad} 
\usepackage{pst-plot} 
\usepackage[space]{grffile} 
\usepackage{etoolbox} 
\makeatletter 
\patchcmd\Gread@eps{\@inputcheck#1 }{\@inputcheck"#1"\relax}{}{}
\makeatother

\hypersetup{pdfpagemode=UseNone}
\usepackage{booktabs}
\usepackage{longtable}
\theoremstyle{plain}
\newtheorem{lemma}{Lemma}[section]
\newtheorem*{theorem*}{Theorem}
\newtheorem*{lemma*}{Lemma}
\newtheorem*{proposition*}{Proposition}
\newtheorem*{conjecture*}{Conjecture}
\newtheorem*{corollary*}{Corollary}
\newtheorem*{problem*}{Problem}
\newtheorem{theorem}[lemma]{Theorem}
\newtheorem{conjecture}[lemma]{Conjecture}
\newtheorem{corollary}[lemma]{Corollary}
\newtheorem{proposition}[lemma]{Proposition}
\newtheorem{claim}[lemma]{Claim}

\newtheorem{question}[lemma]{Question}

\theoremstyle{definition}
\newtheorem{definition}[lemma]{Definition}
\newtheorem{example}[lemma]{Example}
\newtheorem{remark}[lemma]{Remark}

\DeclareMathOperator{\Pic}{Pic}

\DeclareMathOperator{\NS}{NS}

\DeclareMathOperator{\sHom}{\mathcal{H}\kern -.5pt\mathit{om}}
\DeclareMathOperator{\sTor}{\mathcal{T}\kern -1.5pt\mathit{or}}

\begin{document}

\date{\today}

\author[P. Vikash]{Pisya Vikash}
\address{Department of Mathematics, The Pennsylvania State University, University Park, PA 16802}
\email{pmv5172@psu.edu}


\title{Riemann--Roch Polynomials, MBM Classes and Poor IHS Manifolds}

\begin{abstract}
We study poor irreducible holomorphic symplectic manifolds, namely those
containing no rational curves and no codimension-one subvarieties. We show that,
for such manifolds, several natural cones in \(H^{1,1}(X,\mathbb R)\) coincide,
giving strong rigidity consequences. In particular, poor elliptic irreducible
holomorphic symplectic manifolds admit no holomorphic foliations.

Using this rigidity property, we give a criterion for detecting monodromy birationally minimal classes
from the roots of Riemann--Roch polynomials on any irreducible holomorphic symplectic manifold. To each primitive negative class
we associate a polynomial depending only on the deformation type and on its
Beauville--Bogomolov--Fujiki square. Under a simple-root assumption on the
positive real roots, the position of the number \(1\) among these roots implies
that the class is MBM whenever it is of type \((1,1)\). This gives a uniform
root-theoretic sufficient condition for the existence of rational curves on
irreducible holomorphic symplectic manifolds.

We then describe poor elliptic irreducible holomorphic symplectic manifolds in
terms of their deformation spaces and show that every connected component of
the moduli space contains a poor irreducible holomorphic symplectic manifold of
positive Picard rank.
\end{abstract}

\maketitle

\begingroup
  \setcounter{tocdepth}{1}
  \setlength{\parskip}{0pt}
  \tableofcontents
\endgroup

\section{Introduction}\label{sec:intro}
By an \emph{irreducible holomorphic symplectic manifold} we mean a simply connected compact K\"ahler manifold $X$ such that $H^{2,0}(X)=\mathbb{C}\sigma$, where $\sigma$ is a nowhere-vanishing holomorphic symplectic $2$-form. Irreducible holomorphic symplectic manifolds form one of the fundamental building blocks in the Beauville--Bogomolov decomposition of compact K\"ahler manifolds with trivial first Chern class; see \cite{beauville1983varietes}. As higher-dimensional analogues of $K3$ surfaces, they occupy a central position in complex and algebraic geometry because of their rich interplay with Hodge theory, birational geometry, and moduli theory. For such a manifold $X$, a surprisingly large part of its geometry is encoded in the real $(1,1)$-cohomology space $H^{1,1}(X,\mathbb{R})$ together with the Beauville--Bogomolov--Fujiki form $q$, given by
\[
q(\alpha) = \frac{n}{2} \int_X \alpha^2 (\sigma \bar{\sigma})^{n-1} + (1-n)\left(\int_X \alpha \sigma^{n-1}\bar{\sigma}^{n}\right)\left(\int_X \alpha \sigma^{n}\bar{\sigma}^{n-1}\right),
\]
where $\sigma \in H^{2,0}(X)$ is chosen so that $\int_X (\sigma\bar{\sigma})^n=1$.

The birational geometry of an irreducible holomorphic symplectic manifold is governed by certain negative $(1,1)$-classes, called \emph{MBM classes} (short for \emph{monodromy birationally minimal}; see Definition~\ref{def:mbm}). Roughly speaking, the orthogonal hyperplanes to these classes determine the walls of the K\"ahler chamber decomposition of the positive cone. Geometrically, MBM classes should be regarded as the hyperk\"ahler analogues of extremal curve classes: in the projective case, they correspond to extremal rational curves only up to monodromy and birational equivalence, rather than necessarily to extremal rays on the fixed manifold itself. These classes were introduced by Amerik and Verbitsky in \cite{amerik2015rational}; see also the excellent survey \cite{amerik2021rational}. For several specific deformation types, MBM classes, wall divisors, and extremal rays have been studied extensively in works such as \cite{bayer2014mmp, HassettTschinkel2010ExtremalRays, amerik2019mbm, 10.1215/21562261-2081243, Mongardi2023ErratumMonodromy}. An elementry description of them is given in \cite{VIKULOVA2025105349}. However, many of the available arguments remain strongly dependent on the deformation type of the irreducible holomorphic symplectic manifold under consideration.

Let $X$ be an irreducible holomorphic symplectic manifold of dimension $2n$. By \cite{huybrechts1997compact}, there exist constants
\[
a_0,a_2,\dots,a_{2n}\in \mathbb{Q}
\]
such that for every line bundle $L$ on $X$,
\[
\chi(X,L)=\sum_{i=0}^{n}\frac{a_{2i}}{(2i)!}\,q_X(c_1(L))^i,
\]
where $q_X$ denotes the Beauville--Bogomolov--Fujiki quadratic form on $H^2(X,\mathbb{Z})$. The polynomial
\[
RR_X(q):=\sum_{i=0}^{n}\frac{a_{2i}}{(2i)!}\,q^i
\]
is called the \emph{Riemann--Roch polynomial} of $X$. A very recent result of Chen Jiang \cite{Jiang2020PositivityOR} shows that each $a_{2i}$ is a positive real number. Polynomials are explicitly computed for some deformation types in \cite{RiosOrtiz2024RR},\cite{Nie03},\cite{Nie02},\cite{EGL01}.

The main result of this paper is a criterion for detecting MBM classes from the roots of a polynomial naturally associated with the Riemann--Roch polynomial. More precisely
\begin{theorem} \label{Theorem:thmMBMclass}
    Let $\Lambda$ be an integral lattice of IHS manifold type. Suppose that $z \in \Lambda$ is a primitive vector in the lattice with $q(z)<0$. Fix the dimension $2n$. Let $\mathfrak{M^o}(\Lambda)$ be the connected component of the moduli space of marked irreducible holomorphic symplectic manifolds. Let $P_{\mathfrak{M}^o(\Lambda)}^z (x):=RR_X(q(z)x)$. Assume that it has a positive real root and that every root has multiplicity $1$. Let $\{ r_1, r_2, \cdots,r_m\}$ be positive roots of the polynomial $P_{\mathfrak{M}^o(\Lambda)}^z (x)$ with multiplicity $1$, suppose that
    \[
    r_1 < r_2 < \cdots < r_{m-1} <r_m.
    \]
    If
    \begin{itemize}
        \item $n$ even then $m$ is even, and if $ 1\in (r_{1},r_{2}) \cup (r_{3},r_{4}) \cup \cdots \cup (r_{m-1},r_{m})$ then $z$ is MBM, whenever it is of the type $(1,1)$.
        \item $n$ is odd then $m$ is odd, and if $1 \in (0,r_{1}) \cup (r_{2},r_{3}) \cup \cdots \cup (r_{m-1},r_{m})$ then $z$ is MBM, whenever it is of the type $(1,1)$.
    \end{itemize}
\end{theorem}

 This gives an unexpected link between the birational geometry of irreducible holomorphic symplectic manifolds and the root structure of their Riemann--Roch polynomials. The root-theoretic hypotheses appearing in the theorem are satisfied for the currently known deformation types, and are conjectured to hold in general in \cite{Jiang2020PositivityOR}. This gives a partial answer to the question of Amerik--Verbitsky on the existence of MBM classes in lattices of IHS type: under an explicit root-theoretic condition on the Riemann--Roch polynomial, one obtains a primitive negative class which is MBM whenever it is of type $(1,1)$. Consequently, a sufficient condition for the existence of rational curves on irreducible holomorphic symplectic manifolds is given. See the Corollary \ref{cor:rationalcuexc}. Using this we recover some known MBM classes see Examples \ref{EX:1},\ref{EX:2}. The proof of the above theorem is a consequence of cone rigidity of IHS manifolds without codimension one subvarieties and rational curves.

It was observed in \cite{vikash2026classificationpoormanifoldslow} that the most natural examples of compact K\"ahler manifolds with no rational curves and no codimension-one subvarieties are complex tori and irreducible holomorphic symplectic manifolds. Following Zarhin and Bandman \cite{bandman2024jordan}, we call such manifolds \emph{poor}. Several important cones in $H^{1,1}(X,\mathbb{R})$ were studied by Boucksom in \cite{boucksom2004divisorial}. One remarkable feature of poor irreducible holomorphic symplectic manifolds is that these cones coincide; see Proposition~\ref{proposition:propcone} and Corollary~\ref{corollary:concor}. This rigidity makes the study of poor irreducible holomorphic symplectic manifolds particularly interesting.

We also observe that irreducible holomorphic symplectic manifolds of Picard rank $0$ are poor. This naturally leads us to consider irreducible holomorphic symplectic manifolds whose Picard lattice is negative definite with respect to $q_X$; we call such manifolds elliptic irreducible holomorphic symplectic manifolds. We prove that poor elliptic irreducible holomorphic symplectic manifolds admit no holomorphic foliations or fibrations; see Corollary~\ref{corollary:folfib}.

We then apply this poor-IHS framework to the classification of poor elliptic irreducible holomorphic symplectic manifolds. As a first step, we prove that for an elliptic irreducible holomorphic symplectic manifold $X$, the following conditions are equivalent: $X$ contains no rational curves, $X$ contains no curves, and $X$ is poor; see Theorem~\ref{theorem:nocurvespoor}. In view of the examples from Section~3.2, the absence of rational curves is strictly stronger than the absence of codimension-one subvarieties in dimensions greater than $2$. In particular, we construct irreducible holomorphic symplectic manifolds with rational curves but no codimension-one subvarieties. 

Our first classification result describes the points in the universal deformation space of an arbitrary irreducible holomorphic symplectic manifold $X$ that correspond to poor elliptic irreducible holomorphic symplectic manifolds; see Theorem~\ref{thm:maain}. This is enough to describe the corresponding moduli space. In dimension $2$, this result is closely related to \cite[Theorem~3.9]{vikash2026classificationpoormanifoldslow}, where the classification is formulated in terms of the period domain. In higher dimensions, however, that description is no longer sufficient, since the condition of containing no curves is not equivalent to the condition of having no codimension-one subvarieties. Finally, we prove that every connected component of the moduli space of irreducible holomorphic symplectic manifolds contains a poor irreducible holomorphic symplectic manifold of positive Picard rank; see Theorem~\ref{thm:picneq0}.

The preceding results reduce the classification of poor irreducible holomorphic
symplectic manifolds to the problem of excluding the parabolic case. More
precisely, once one knows that every irreducible holomorphic symplectic manifold
with no rational curves and no codimension-one subvarieties is either of Picard
rank zero or elliptic, the classification follows from the results here.

Campana introduced and studied the notion of simple compact Kähler manifolds in several works; see for example \cite{campana1983densite,Campana2004Orbifolds,campana2006isotrivialite,Campana2011Orbifoldes,CampanaDemaillyVerbitsky2014,Campana2026Bogomolov}.
\begin{definition}
    A compact K\"ahler manifold $X$ is simple if a very general point of $X$ lies on no nontrivial subvariety. 
\end{definition}

Then a result of Fujiki implies that poor IHS manifolds are simple (see,  \cite[proposition 5.16]{fujiki1987derham}). However, Examples \ref{exp:AVexample} imply that simple dosent imply poor and also do not satisfy the cone rigidity that we have established for Poor manifolds. Therefore, poor manifolds are more rigid than simple manifolds in case of Irreducible holomorphic symplectic manifolds. In \cite{CampanaHoringPeternell2016,CampanaOguisoPeternell2010}, authors proposed the following conjecture.
\begin{conjecture}
    Irreducible holomorphic symplectic manifold of algebraic dimension $0$ is elliptic.
\end{conjecture}
This is much stronger than what we need to complete the classification. In fact in \cite[Theorem D]{horing2025nonvanishing} together with Proposition \ref{proposition:propcone} will imply Irreducible holomorphic symplectic with no curves and codimension one sub varieties is elliptic.

\section*{Notation and conventions.}
We work over $\mathbb{C}$.

\begin{itemize}
    \item We write $\mathbb{Z}$, $\mathbb{Q}$, $\mathbb{R}$, and $\mathbb{C}$ for the integers, rational numbers, real numbers, and complex numbers, respectively.

    \item For a compact complex manifold $X$, we denote by $\operatorname{Pic}(X)$ the Picard group of $X$, by $TX$ the holomorphic tangent bundle, by $\Omega_X^1$ the cotangent bundle, and by $K_X$ the canonical bundle. We write
    \[
    h^{p,q}(X):=\dim H^{p,q}(X).
    \]

    \item An \emph{irreducible holomorphic symplectic manifold} (abbreviated \emph{IHS manifold}) is a simply connected compact K\"ahler manifold $X$ such that
    \[
    H^{2,0}(X)=\mathbb{C}\sigma
    \]
    for a nowhere-vanishing holomorphic symplectic form $\sigma$.

    \item If $\dim X=2n$, we denote by $q_X$ (or simply $q$, when no confusion is possible) the Beauville--Bogomolov--Fujiki form on $H^2(X,\mathbb{Q})$.

    \item For an IHS manifold $X$, we use the following notation for cones in $H^{1,1}(X,\mathbb{R})$:
    \[
    \mathcal{K}_X \text{ (K\"ahler cone)}, \quad \mathcal{N}_X:=\overline{\mathcal{K}_X} \text{ (nef cone)}, \quad
    \mathcal{P}_X \text{ (pseudo-effective cone)},
    \]
    \[
    \mathcal{C}_X \text{ (positive cone with respect to $q_X$)}, \quad \mathcal{BK}_X \text{ (birational K\"ahler cone)},
    \]
    \[
    \mathcal{MK}_X \text{ (modified K\"ahler cone)}, \quad \mathcal{MN}_X \text{ (modified nef cone)}.
    \]
    We denote by $\check{K}_X$ the dual K\"ahler cone with respect to $q_X$.

    \item A nonzero rational class $z\in H^{1,1}(X,\mathbb{Q})$ is called \emph{monodromy birationally minimal} (MBM) if it satisfies Definition~2.4. We denote by $\operatorname{Mon}(X)$ the monodromy group of $X$.

    \item A compact connected complex manifold is called \emph{poor} if it contains no rational curves and no codimension-one subvarieties. An IHS manifold is called \emph{elliptic} if its Picard lattice is negative definite with respect to $q_X$.

    \item If $T$ is a closed positive current, we represent its cohomology class by $\{T\}$ and write $\nu(T,x)$ for the Lelong number of $T$ at a point $x\in X$, and $\nu(T,Y)$ for the generic Lelong number of $T$ along an irreducible analytic subset $Y\subset X$.

    \item We use the standard notation
    \[
    d^c:=\frac{i}{2\pi}(\overline{\partial}-\partial),
    \qquad\text{so that}\qquad
    dd^c=\frac{i}{\pi}\partial\overline{\partial}.
    \]
    When working with a K\"ahler form, we often denote it by $\beta$.

    \item If $X$ is an IHS manifold of dimension $2n$, then by Hirzebruch--Riemann--Roch,
    \[
    \chi(X,L)=\sum_{i=0}^{n}\frac{a_{2i}}{(2i)!}\,q_X(c_1(L))^i.
    \]
    We write
    \[
    RR_X(t):=\sum_{i=0}^{n}\frac{a_{2i}}{(2i)!}\,t^i
    \]
    for the associated Riemann--Roch polynomial.
\end{itemize}
\section{Cones in poor IHS manifolds and MBM--Riemann--Roch criterion}
\subsection{Preliminaries}
\begin{definition} \label{definition:IHS}
    A compact K\"ahler manifold $X$ is called an irreducible holomorphic symplectic manifold, or an IHS manifold, if it is simply connected, $H^2(X,\mathcal{O}_X) \cong H^{0,2}(X)$ is one dimensional and $H^{2,0}(X)$ is generated by nowhere- vanishing holomorphic symplectic $2$ form $\sigma$.
\end{definition}
This definition is motivated by the Beauville–Bogomolov decomposition theorem. Basic results about IHS manifolds can be found in \cite{huybrechts1997compact} and \cite{HuybrechtsErratum}. Let $X$ be an IHS manifold. Beauville constructed a non-degenerate symmetric bilinear form on $H^2(X,\mathbb{Q})$ (see \cite{beauville1983varietes}). Denoted by $q$, this form is given by:
\[
q_X(\alpha) = \frac{n}{2} \int_X \alpha^2 (\sigma \bar{\sigma})^{n-1} + (1 - n) \left( \int_X \alpha \sigma^{n-1} \bar{\sigma}^n \right) \left( \int_X \alpha \sigma^n \bar{\sigma}^{n-1} \right),
\]
where $\sigma \in H^{2,0}(X)$ is chosen such that $\int_X (\sigma \bar{\sigma})^n = 1$. This form has signature $(3,b_2(X)-3)$. It can also be written as:
\[
q(\alpha,\beta) = c\int_X \alpha \wedge \beta \wedge (\sigma\bar{\sigma}) ^{n-1}
\]
for some constant $c$. We now recall the definitions of the cones in $H^{1,1}(X,\mathbb{R})$ that will be used throughout the paper. Let $X$ be a compact K\"ahler manifold.
\begin{enumerate}
    \item The K\"ahler cone $\mathcal{K}_X$ is the cone consisting of cohomology classes of all K\"ahler forms on $X$. The cone $\mathcal{K}_X$ is open, and its closure is called the nef cone, denoted by $\mathcal{N}_X$.

    \item The pseudo-effective cone $\mathcal{P}_X \subset H^{1,1}(X,\mathbb{R})$ consists of the cohomology classes of all closed positive $(1,1)$-currents on $X$ (Positivity in the sense of currents). It is known that $\mathcal{P}_X$ is closed.

    \item Now assume that $X$ is an irreducible holomorphic symplectic manifold. The positive cone
    $\mathcal{C}_X \subset H^{1,1}(X,\mathbb{R})$ is the connected component of the cone
    \[
        \{ v \in H^{1,1}(X,\mathbb{R}) \mid q(v) > 0 \}
    \]
    that contains the cohomology class of a K\"ahler form.

    \item Again assume that $X$ is an IHS manifold. The birational K\"ahler cone $\mathcal{BK}_X$ is the cone consisting of cohomology classes of the form $f^*[\omega]$, where $f \colon X \dashrightarrow X'$ is a bimeromorphic map to a hyperk\"ahler manifold $X'$ and $\omega$ is a K\"ahler form on $X'$. Note that the cone $\mathcal{BK}_X$ is not necessarily convex, but its closure is convex.
\end{enumerate}

Before giving the definitions of two more types of classes we are interested in, let us fix some notation. Let $T$ be a positive closed current of bidimension $(p,p)$ on an open set $\Omega\subset \mathbb{C}^n$, and let $a\in \Omega$. Set
\[
d^c:=\frac{i}{2\pi}(\bar\partial-\partial), \qquad \beta:=dd^c|z|^2.
\]
The \emph{Lelong number} of $T$ at $a$ is defined by
\[
\nu(T,a):=\lim_{r\to 0}\frac{1}{r^{2p}}\int_{B(a,r)} T\wedge \beta^p.
\]
The existence of the limit follows from the Lelong--Jensen formula (see \cite[Chapter 2.B]{demailly2001multiplier}, also \cite{Lel57}). We denote by $\nu(T,x)$ its Lelong number at a point $x \in X$. For any analytic subset $Y \subset X$, the \emph{generic Lelong number} of $T$ along $Y$ is defined by
\[
    \nu(T,Y) := \inf\{\nu(T,x) \mid x \in Y\}.
\]
Then $\nu(T,Y) = \nu(T,x)$ for a very general point $x \in Y$.
Let $\alpha$ be a class in $H^{1,1}_{BC}(X,\mathbb{R})$ (Bott-Chern cohomology).

\begin{enumerate}
    \item The class $\alpha$ is said to be a \emph{modified K\"ahler class} if and only if it contains a K\"ahler current $T$ such that
    \[
        \nu(T,D) = 0 \quad \text{for every prime divisor } D \subset X.
    \]
    We denote the \emph{modified K\"ahler cone} by $\mathcal{MK}_X$.

    \item The class $\alpha$ is said to be a \emph{modified nef class} if and only if, for every $\varepsilon > 0$, there exists a closed $(1,1)$-current $T_\varepsilon \in \alpha$ such that
    \[
        T_\varepsilon \geq -\varepsilon \,\omega
        \quad\text{and}\quad
        \nu(T_\varepsilon,D) = 0 \quad \text{for every prime divisor } D \subset X,
    \]
    where $\omega$ is a fixed K\"ahler form on $X$. We denote the \emph{modified nef cone} by $\mathcal{MN}_X$.
\end{enumerate}

In particular, one has the inclusions
\[
    \mathcal{K}_X \subset \mathcal{MK}_X \subset \mathcal{MN}_X \subset \mathcal{P}_X,
    \qquad
    \overline{\mathcal{K}_X} = \mathcal{N}_X \subset \mathcal{MN}_X.
\]
When $X$ is an IHS manifold, from \cite{boucksom2004divisorial} and \cite{huybrechts2003kahler} one also has $\overline{\mathcal{BK}_X} = {\mathcal{MN}_X \subset \overline{\mathcal{C}_X}}$,
\begin{center}
\begin{tikzcd}[
    column sep=1em,
    row sep=1em,
    >={Stealth[length=0.5pt,width=0.5pt]} 
]
    \mathcal{K}_X \arrow[r, hook] \arrow[d, hook] &
    \mathcal{BK}_X \arrow[r, hook] &
    \mathcal{C}_X \\[0.2cm]
    \mathcal{N}_X = \overline{\mathcal{K}_X} \arrow[r, hook] &
    \mathcal{MN}_X = \overline{\mathcal{BK}_X} \arrow[r, hook] &
    \mathcal{P}_X.
\end{tikzcd}
\end{center}
Define the dual Kähler cone $\check{\mathcal{K}}_X \subset H^{1,1}(X, \mathbb{R})$ as
\[
    \check{\mathcal{K}}_X := \left\{ x \in H^{1,1}(X, \mathbb{R}) \,\middle|\, \forall y \in \mathcal{K}_X,\ q(x, y) > 0 \right\}.
\]
It is convex cone. Since the Beauville–Bogomolov form is positive on the product of two Kähler classes, we have $\check{\mathcal{K}}_X \supset \mathcal{K}_X$. Denote by $\overline{\check{\mathcal{K}}_X}$ the closure of $\check{\mathcal{K}}_X$ in $H^{1,1}(X, \mathbb{R})$, and by $-\check{\mathcal{K}}_X$ the opposite cone. Clearly,
\[
    \overline{\check{\mathcal{K}}_X}
    = \left\{ x \in H^{1,1}(X, \mathbb{R}) \,\middle|\, \forall y \in \mathcal{K}_X,\ q(x, y) \geq 0 \right\}
\]
and
\[
    -\overline{\check{\mathcal{K}}_X}
    = \left\{ x \in H^{1,1}(X, \mathbb{R}) \,\middle|\, \forall y \in \mathcal{K}_X,\ q(x, y) \leq 0 \right\}.
\]

\begin{theorem}[\cite{verbitsky2007quaternionic}]\label{theorem:cohom}
    Let $X$ be an IHS manifold, and $L$ a holomorphic line bundle on $X$ with $c_1(L) \neq 0$. Then exactly one of the following holds.
    \begin{enumerate}
        \item $c_1(L) \in \overline{\check{\mathcal{K}}_X}$; then $H^i(X,L) = 0$ for all $i > \frac{\dim_{\mathbb{C}} X}{2}$.
        \item $c_1(L) \in -\overline{\check{\mathcal{K}}_X}$; then $H^i(X,L) = 0$ for all $i < \frac{\dim_{\mathbb{C}} X}{2}$.
        \item $c_1(L)$ does not lie in $-\overline{\check{\mathcal{K}}_X} \cup \overline{\check{\mathcal{K}}_X}$; then $H^i(X,L) = 0$ for all $i \neq \frac{\dim_{\mathbb{C}} X}{2}$.
    \end{enumerate}
\end{theorem}

 The notion of MBM classes was introduced by Ekaterina Amerik and Misha Verbitsky. MBM classes are used throughout the text. Let us recall the definition of MBM classes together with some results on them.
\begin{definition}{\cite{markman2020beauville}}
Let $X$ be an irreducible holomorphic symplectic manifold. An automorphism
\[
g \in \operatorname{GL}(H^*(X,\mathbb{Z}))
\]
is called a \emph{monodromy operator} if there exist
\begin{itemize}
    \item a smooth proper holomorphic family
    \[
    \pi:\mathcal{X}\to B
    \]
    of irreducible holomorphic symplectic manifolds,
    \item a point $b_0\in B$ such that $\mathcal{X}_{b_0}\cong X$,
    \item and a loop $\gamma:[0,1]\to B$ based at $b_0$,
\end{itemize}
such that parallel transport along $\gamma$ induces the automorphism
\[
g:H^*(X,\mathbb{Z})\to H^*(X,\mathbb{Z}).
\]
The \emph{monodromy group} of $X$, denoted by $\operatorname{Mon}(X)$, is the subgroup of
\[
\operatorname{GL}(H^*(X,\mathbb{Z}))
\]
generated by all monodromy operators.
\end{definition}
\begin{definition}\label{def:mbm}{\cite{amerik2015rational}}
    A nonzero rational class $z \in H^{1,1}(X,\mathbb{Q})$ is called monodromy birationally minimal, or MBM, if there exists an isometry $\gamma \in O(H^2(X, \mathbb{Z}))$ belonging to the monodromy group such that $\gamma(z)^{\perp} \subset H^{1,1}(X)$ contains a face of the K\"ahler cone of one of the birational models $X'$ of $X$.
\end{definition}
Let us recall a result from \cite{amerik2015rational}.
\begin{theorem} \cite{amerik2015rational}\label{theorem:mbmposk}
    Let $X$ be an IHS manifold, and $S \subset H^{1,1}(X)$ the set of all MBM classes in $H^{1,1}(X)$. Consider the corresponding set of hyperplanes $S^\perp := \{W = z^\perp \mid z \in S\}$ in $H^{1,1}(X)$. Then the K\"ahler cone of $X$ is a connected component of $\mathcal{C}_X \setminus S^\perp$.
\end{theorem}
By a theorem of Ran \cite[Corollary 5.2]{ran1995hodge}, if $X$ is a holomorphic symplectic manifold of dimension $2n$ and $C\subset X$ is a rational curve, then every irreducible component of the deformation space of $C$ in $X$ has dimension at least $2n-2$ at $[C]$.
\begin{definition}
Let $X$ be a holomorphic symplectic manifold of dimension $2n$. A rational curve $C\subset X$ is called \emph{minimal} if every irreducible component of the local deformation space of $C$ in $X$ has dimension $2n-2$ at $[C]$.
\end{definition}

\begin{remark}\label{remark:bbf-dual-curve-classes}
Let $X$ be an IHS manifold of dimension $2n$. The Beauville--Bogomolov--Fujiki form
\[
q_X : H^2(X,\mathbb Q) \times H^2(X,\mathbb Q) \longrightarrow \mathbb Q
\]
is nondegenerate. Hence it identifies $H^2(X,\mathbb Q)$ with its dual $H^2(X,\mathbb Q)^\vee$. By Poincaré duality,
\[
H^2(X,\mathbb Q)^\vee \cong H_2(X,\mathbb Q).
\]
Thus the BBF form gives an isomorphism
\[
H^2(X,\mathbb Q) \cong H_2(X,\mathbb Q).
\]
Equivalently, for every curve class $R \in H_2(X,\mathbb Q)$ there is a unique class $R^\vee \in H^2(X,\mathbb Q)$ such that
\[
\alpha \cdot R = q_X(\alpha, R^\vee) \qquad \text{for all } \alpha \in H^2(X,\mathbb Q).
\]
We call $R^\vee$ the \emph{BBF-dual} of $R$. If $C \subset X$ is a curve, we define
\[
q_X([C]) := q_X([C]^\vee),
\]
where $[C]^\vee$ is the BBF-dual of the homology class $[C] \in H_2(X,\mathbb Q)$. Equivalently, one may regard $[C]$ via its Poincaré dual class in $H^{4n-2}(X,\mathbb Q)$ and transport the BBF form to this space by duality. In summary, the curve class, its Poincaré dual, and its BBF-dual are related as follows:
\[
\begin{array}{c}
H^2(X,\mathbb Q) \;\xrightarrow[\alpha\mapsto q_X(\alpha,-)]{\sim}\; H^2(X,\mathbb Q)^\vee \;\xrightarrow[\operatorname{PD}^{-1}]{\sim}\; H_2(X,\mathbb Q)
\end{array}
\]
Thus, for a curve $C \subset X$,
\[
[C] \in H_2(X,\mathbb Q),\qquad
\operatorname{PD}[C] \in H^{4n-2}(X,\mathbb Q),\qquad
[C]^\vee \in H^2(X,\mathbb Q),
\]
where $PD$ stands for Poincare dual and
\[
q_X([C]) = q_X([C]^\vee).
\]
\end{remark}

\begin{definition}[Full MBM locus]
Let $z\in H^2(M,\mathbb Q)$ be an MBM class. The full MBM locus of $z$ is the union of all MBM curves whose BBF-dual classes are proportional to $z$, together with their degenerations.
\end{definition}
\begin{theorem}[\cite{amerik2021contraction}]\label{thm:AV-local-MBM}
Let $X$ be an IHS manifold with $b_2(X)>5$, and let $z\in H^2(X,\mathbb Q)$ be a class with
\[
q(z)<0.
\]
Assume that $z$ is represented by a minimal rational curve on $X$. Let
\[
\mathcal X\longrightarrow B:=\operatorname{Def}(X)
\]
be the Kuranishi family of $X$, with central fiber $\mathcal X_0=X$. After shrinking $B$, we regard $z$ as a flat rational cohomology class on all nearby fibers by parallel transport. Define the local Hodge locus of $z$ by
\[
B_z:=\{t\in B:\ z_t\in H^{1,1}(\mathcal X_t,\mathbb Q)\}.
\]
Then, after shrinking $B_z$ around $0$, the full MBM loci of the classes $z_t$ are real-analytically isomorphic for all $t\in B_z$.
\end{theorem}
\subsection{Cones in poor IHS manifolds }
In this subsection, we discuss the structure of the cones defined above and prove some consequences.
\begin{definition}
    Let $X$ be a compact K\"ahler manifold. The negative part of a pseudo-effective class $\alpha \in H^{1,1 }_{\partial\bar{\partial}}(X, \mathbb{R})$ is defined as $N(\alpha) := \sum \nu(\alpha, D)[D]$. The Zariski projection of $\alpha$ is $Z(\alpha) := \alpha - \{N(\alpha)\}$. We call the decomposition $\alpha = Z(\alpha) + \{N(\alpha)\}$ the divisorial Zariski decomposition of $\alpha$.
\end{definition}
It is shown in \cite{boucksom2004divisorial} that $Z(\alpha) \in \mathcal{MN}_X$ (Proposition 3.8) and that $N(\alpha)$ is effective (Proposition 3.11).
\begin{proposition} \label{proposition:propcone}
    Let $X$ be a poor irreducible holomorphic symplectic manifold. Then
    \[
        \mathcal{P}_X = \mathcal{N}_X.
    \]
\end{proposition}

\begin{proof}
    Since $X$ is poor, by \cite[Corollary~3.3]{huybrechts2003kahler} we have
    \[
        \mathcal{C}_X = \mathcal{K}_X,
    \]
    hence
    \[
        \mathcal{K}_X = \mathcal{BK}_X = \mathcal{C}_X.
    \]

    Let $\alpha$ be a pseudoeffective class. By the divisorial Zariski decomposition, we can write
    \[
        \alpha = Z(\alpha) + \{N(\alpha)\},
    \]
    where $Z(\alpha)$ is the positive (or “movable”) part and $N(\alpha)$ is the negative part, supported on divisors. Since $X$ is poor, there are no divisors, and thus
    \[
        \{N(\alpha)\} = 0.
    \]
    Therefore, $\alpha = Z(\alpha)$ is modified nef, see \cite[Proposition 3.8]{boucksom2004divisorial}. In other words,
    \[
        \mathcal{P}_X \subset \mathcal{MN}_X.
    \]
    The reverse inclusion $\mathcal{MN}_X \subset \mathcal{P}_X$ holds by definition, since modified nef classes are pseudo\-effective. Hence
    \[
        \mathcal{P}_X = \mathcal{MN}_X.
    \]

    Finally, by the usual equalities for an IHS manifold,
    \[
        \mathcal{MN}_X = \overline{\mathcal{BK}_X} = \overline{\mathcal{C}_X} = \mathcal{N}_X,
    \]
    so $\mathcal{P}_X = \mathcal{N}_X$ as claimed.
\end{proof}

\begin{corollary}\label{corollary:concor}
    If $X$ is poor, then
    \[
        \mathcal{MN}_X = \overline{\mathcal{BK}_X}
        = \overline{\mathcal{C}_X}
        = \mathcal{N}_X
        \quad\text{and}\quad
        \mathcal{MK}_X = \mathcal{K}_X.
    \]
\end{corollary}

\begin{proof}
    The first chain of equalities follows from Proposition~\ref{proposition:propcone} and the identities
    \[
        \mathcal{MN}_X = \overline{\mathcal{BK}_X}
        \quad\text{and}\quad
        \overline{\mathcal{C}_X} = \mathcal{N}_X.
    \]
    For the second statement, note that modified K\"ahler classes are big, hence lie in the positive cone $\mathcal{C}_X$. Since $X$ is poor, we have $\mathcal{C}_X = \mathcal{K}_X$, so any modified K\"ahler class lies in the K\"ahler cone. Thus
    \[
        \mathcal{MK}_X \subset \mathcal{K}_X.
    \]
    The reverse inclusion $\mathcal{K}_X \subset \mathcal{MK}_X$ is clear because every K\"ahler class is, in particular, modified K\"ahler. Therefore $\mathcal{MK}_X = \mathcal{K}_X$.
\end{proof}
Using this description of cones, we can arrive at some interesting conclusions.
\begin{proposition}\label{prop:negative-square-cohom}
    Let $X$ be a poor IHS manifold of complex dimension $2n$. Let $L$ be a line bundle on $X$ such that $q\big(c_1(L)\big) < 0$. Then
    \[
        H^i(X,L) = 0 \quad \text{for all } i \neq n,
    \]
    and in particular
    \[
        \dim H^n(X,L) = (-1)^n \chi(X,L).
    \]
\end{proposition}

\begin{proof}
        Since $X$ is poor, Corollary~\ref{corollary:concor} implies that every class in the closure of the positive cone $\overline{\mathcal{C}}_X$ can be approximated by Kähler classes. More precisely,
        \[
            \overline{\mathcal{K}}_X = \overline{\mathcal{C}}_X.
        \]
        The closed dual cone only depends on the closure of the cone we start with, so
        \[
            \overline{\check{\mathcal{K}}_X}
            = \left\{ x \,\middle|\, \forall y \in \overline{\mathcal{K}}_X,\ q(x,y) \ge 0 \right\}
            = \left\{ x \,\middle|\, \forall y \in \overline{\mathcal{C}}_X,\ q(x,y) \ge 0 \right\}
            = \overline{\check{\mathcal{C}}_X}.
        \]
        On the other hand, the positive cone $\mathcal{C}_X$ is self-dual with respect to $q$, so
        \[
            \overline{\check{\mathcal{C}}_X} = \overline{\mathcal{C}}_X.
        \]
        Combining these equalities, we obtain
        \[
            \overline{\check{\mathcal{K}}_X} = \overline{\mathcal{C}}_X.
        \]

        By definition of the positive cone, $q(x) > 0$ for all $x \in \mathcal{C}_X$, and hence
        \[
            q(x) \ge 0 \quad \text{for all } x \in \overline{\mathcal{C}}_X = \overline{\check{\mathcal{K}}_X},
        \]
        with equality only on the boundary where $x$ is isotropic. The same holds for $-x$ on $-\overline{\check{\mathcal{K}}_X}$, since $q(-x) = q(x)$.

        Now assume $q\big(c_1(L)\big) < 0$. Then both $c_1(L)$ and $-c_1(L)$ have strictly negative square with respect to $q$, so neither $c_1(L)$ nor $-c_1(L)$ can lie in $\overline{\check{\mathcal{K}}_X}$ (or in $-\overline{\check{\mathcal{K}}_X}$). In other words,
        \[
            c_1(L) \notin \overline{\check{\mathcal{K}}_X} \cup \big(-\overline{\check{\mathcal{K}}_X}\big).
        \]
        Thus $c_1(L)$ falls into case~(3) of Theorem~\ref{theorem:cohom}. Since $\dim_{\mathbb{C}} X = 2n$, we conclude that 
        \[
            H^i(X,L) = 0 \quad \text{for all } i \neq n.
        \]

        Finally, by the definition of the Euler characteristic,
        \[
            \chi(X,L) = \sum_{i=0}^{2n} (-1)^i \dim H^i(X,L)
            = (-1)^n \dim H^n(X,L),
        \]
        so $\dim H^n(X,L) = (-1)^n \chi(X,L)$, as claimed.
\end{proof}
\subsection{MBM--Riemann--Roch criterion}
We are now ready to prove the main result. First, let us fix some notation and some facts about the period map for IHS manifold. Let $\Lambda$ be a lattice of signature $(3,b-3)$, where $b\ge 3$. Define
\[
\mathfrak{M}_{\Lambda}
:= \{(X,\varphi)\}/\sim,
\]
where $(X,\varphi)$ is a marked irreducible symplectic manifold (we usually also fix the dimension $2n$), and where $(X,\varphi)\sim (X',\varphi')$ if and only if there exists an isomorphism $f\colon X \xrightarrow{\sim} X'$ such that
\[
f^{*} = \pm\bigl(\varphi^{-1}\circ \varphi'\bigr).
\]
The Local Torelli Theorem (\cite{beauville1983varietes}) allows one to patch the local charts $\operatorname{Def}(X)$. Thus, $\mathfrak{M}_{\Lambda}$ carries the structure of a non-separated (i.e.\ non-Hausdorff) complex manifold. The period map can be considered as a holomorphic map
\[
\mathcal{P}\colon \mathfrak{M}_{\Lambda}\longrightarrow Q_{\Lambda} := \{x \ | \ x\cdot x =0 , \ x\cdot \bar{x}>0\} \subset \mathbf{P}(\Lambda_{\mathbb{C}}).
\]
Let $\mathfrak{M}^0_{\Lambda}$ denote a connected component of $\mathfrak{M}_{\Lambda}$. Then the period map is surjective when restricted to $\mathfrak{M}^0_{\Lambda}$ (\cite{HuybrechtsErratum}, \cite{huybrechts1997compact}).
\begin{definition}
    An integral lattice $\Lambda$ is called a lattice of IHS manifold type if there exists an irreducible holomorphic symplectic manifold $X$ with an isomorphism
    \[
        \phi: H^2(X,\mathbb{Z}) \xrightarrow{\cong} \Lambda.
    \]
\end{definition}

Let $\Lambda$ be an integral lattice of IHS manifold type. Suppose $z$ is a primitive vector in $\Lambda$. Then consider the real hyperplane
\[
z^{\perp} := \{ \sigma \in \Lambda_{\mathbb{R}} \mid q(\sigma,z) = 0 \}.
\]
The restriction of $q$ to $z^{\perp}$ has the signature $(3,b_2-4)$ or $(2,b_2-3)$. In particular, $z^{\perp}$ contains positive $2$-planes, and hence
\[
Q_{\Lambda} \cap \mathbb{P}(z^{\perp} \otimes \mathbb{C}) \neq \varnothing.
\]
By surjectivity of the period map on the connected component $\mathfrak{M}^0_{\Lambda}$, there exists a marked IHS manifold $(Y,\psi) \in \mathfrak{M}^0_{\Lambda}$ such that
\[
\mathcal{P}([(Y,\psi)]) \in
Q_{\Lambda} \cap \mathbb{P}(z^{\perp} \otimes \mathbb{C}).
\] 
From the exponential sequence and the simply connectedness of IHS manifolds, there exists  a line bundle $L$ on $Y$ such that $c_1(L) = z$, under the identification $H^2(Y,\mathbb{Z}) \simeq \Lambda$. Then we have
\[
\chi(Y,L) = \sum_{i=0}^n \frac{a_i}{(2i)!}q(c_1(L))^i.
\]
Then by the Fujiki relations, the constants $a_i$ depend only on the topology of $Y$. Therefore, they are independent of the choice $(Y,\psi) \in \mathfrak{M}^o(\Lambda)$. We have  
\[
\chi(Y,L^{\otimes m}) = \sum_{i=0}^n \frac{a_i}{(2i)!}q(z)^i m^{2i}.
\] 
Now, define
\[
P_{\mathfrak{M}^o(\Lambda)}^z (x) := \sum_{i=0}^n \frac{a_i}{(2i)!}q(z)^i x^{i},
\]
\[
RR_X(t):=\sum_{i=0}^n\frac{a_i}{(2i)!}t^i.
\]
The last polynomial is called the Riemann--Roch polynomial. Observe that $P_{\mathfrak{M}^o(\Lambda)}^z (0) = n+1$. Therefore, $0$ is never a root. In \cite{Jiang2020PositivityOR}, it was observed that if a real root exists for $RR(t)$, then it has to be a negative number.

\begin{proof}[Proof of Theorem \ref{Theorem:thmMBMclass}]
By the surjectivity of the period map, we can choose $(Y,\psi)$ such that there exists a line bundle $L$ such that $\psi(c_1(L)) =z$. If the Picard rank is not equal to $1$, then we can deform $(Y,\psi)$ to $(Y_0,\psi_0)$ such that it has Picard rank $1$ and $z$ is of type $(1,1)$ and there exists a line bundle $L_0$ on $Y_0$ such that $\psi_0(c_1(L_0)) =z$ (for details, see \cite{huybrechts1997compact}). Since $\rho(Y_0)=1$ and $q(z)<0$, the lattice $\operatorname{Pic}(Y_0)\otimes \mathbb Q$ is generated by a class of negative Beauville--Bogomolov square. Hence $\operatorname{Pic}(Y_0)$ is negative definite, so $Y_0$ is elliptic. Observe that if $z$ is MBM on $Y_0$, then it is MBM whenever it is of type $(1,1)$; due to \cite[Corollary 5.13]{amerik2015rational}. Assume for a contradiction that $z$ is not MBM. Since the Picard rank is $1$ and $z$ generates $H^{1,1}(Y_0,\mathbb Q)$, every rational $(1,1)$-class on $Y_0$ is proportional to $z$. As the MBM property is invariant under multiplication by a nonzero rational number, the assumption that $z$ is not MBM implies that $Y_0$ has no MBM classes. We make the following claim.
    \begin{claim}
        If $Y_0$ is an elliptic IHS manifold with Picard rank $1$ and does not have any MBM classes then $Y_0$ is poor.
    \end{claim}
    \begin{proof}
        Suppose that $Y_0$ contains a rational curve $C$. Then its cohomology class satisfies
        \[
        [C]\in H^{2n-1,2n-1}(Y_0,\mathbb{Q}).
        \]
        By Remark~\ref{remark:bbf-dual-curve-classes}, the class $[C]$ corresponds to a rational $(1,1)$-class on $Y_0$. Since $\rho(Y_0)=1$ and $z$ is of type $(1,1)$, the space $H^{1,1}(Y_0,\mathbb{Q})$ is generated by $z$. Hence there exists $r\in\mathbb{Q}_{>0}$ such that the class corresponding to $[C]$ is $rz$. In particular, $\pm z$ is $\mathbb{Q}$-effective (in the sense of \cite[Definition 1.12]{amerik2015rational}). Therefore, by \cite[Theorem~5.10]{amerik2015rational}, the class $z$ is MBM, a contradiction. If, instead, it contains a codimension-one subvariety but no rational curves, let $D$ be an irreducible component. Let $\mathcal{O}_{Y_0}(D)$ be the associated line bundle. $Y_0$ is elliptic, so $q(c_1(\mathcal{O}_{Y_0}(D)))<0$, then by \cite[Proposition 4.7]{boucksom2004divisorial}, $Y_0$ admits a rational curve. Again, we have a contradiction.
    \end{proof}
    By the above claim, $Y_0$ is poor. By the definition of $P_{\mathfrak{M}^o(\Lambda)}^z(x)$ and the Riemann--Roch formula, one has
    \[
    P_{\mathfrak{M}^o(\Lambda)}^z(k^2)=\chi(Y_0,L_0^{\otimes k})
    \]
    for every integer $k$. Since $q(c_1(L_0^{\otimes k}))=k^2q(z)<0$ for $k\neq 0$, Proposition~\ref{prop:negative-square-cohom} gives
    \[
    \chi(Y_0,L_0^{\otimes k})=(-1)^n h^n(Y_0,L_0^{\otimes k}) \qquad\text{for all }k\neq 0.
    \]
    Also,
    \[
    P_{\mathfrak{M}^o(\Lambda)}^z(0)=\chi(Y_0,\mathcal O_{Y_0})=n+1>0.
    \] 
    So the sign of the polynomial behaves in the following way.
\newcommand{\signpicture}[2]{
\resizebox{#2\linewidth}{!}{%
\begin{tikzpicture}[>=Stealth, x=1.2cm, y=0.8cm, baseline={(current bounding box.center)}]
  \coordinate (O)  at (0,0);
  \coordinate (g1) at (2,0);
  \coordinate (g2) at (4,0);
  \coordinate (gn) at (7,0);
  \coordinate (R)  at (9.8,0);

  \draw[thick,->] (O) -- (R);

  \foreach \P/\lab in {O/0,g1/r_1,g2/r_2,gn/r_m}{
    \draw (\P) circle (2pt);
    \node[below] at (\P) {$\lab$};
  }

  \path (O)  -- (g1) node[midway,above] {$+$};
  \path (g1) -- (g2) node[midway,above] {$-$};
  \path (g2) -- (gn) node[midway,above] {$\cdots$};
  \path (g2) -- (gn) node[midway,below] {$\cdots$};
  \path (gn) -- (R)  node[midway,above] {$#1$};

  \node[right] at (R) {$.$};
\end{tikzpicture}}%
}

\begin{center}
\renewcommand{\arraystretch}{1.2}
\setlength{\tabcolsep}{10pt}
\begin{tabular}{|c|c|}
\hline
$m$ is even. & $m$ is odd. \\
\hline
\signpicture{+}{0.45} & \signpicture{-}{0.45} \\
\hline
\end{tabular}
\end{center}
Now, we look at each case.
\begin{description}
    \item[Case 1] If $n,m$ are even, then we have $P_{\mathfrak{M}^o(\Lambda)}^z (1) \geq 0$. If $ 1\in (r_{1},r_{2}) \cup (r_{3},r_{4}) \cup \cdots \cup (r_{m-1},r_{m})$, we have a contradiction from the above table.
    \item[Case 2] $n$ is even and $m$ is odd. This case cannot occur. Indeed, if $x\in \mathbb{R}$ is a root of $P_{\mathfrak{M}^o(\Lambda)}^z(x)$, then $q(z)x$ is a real root of the corresponding Riemann--Roch polynomial. By \cite{Jiang2020PositivityOR}, every real root of the Riemann--Roch polynomial is negative. Since $q(z)<0$, every real root of $P_{\mathfrak{M}^o(\Lambda)}^z(x)$ is positive. Thus $m$ is the total number of real roots. The remaining $n-m$ roots are nonreal and occur in complex conjugate pairs, so $n-m$ must be even. This contradicts the fact that $n$ is even and $m$ is odd.
    \item[Case 3] $n,m$ are odd, then we have $P_{\mathfrak{M}^o(\Lambda)}^z (1) \leq 0$. If $1 \in (0,r_{1}) \cup (r_{2},r_{3}) \cup \cdots \cup (r_{m-1},r_{m})$, we have a contradiction from the above table.
    \item[Case 4] $n$ is odd and $m$ is even. This case cannot occur for the same reason as Case 2: the nonreal roots occur in conjugate pairs, so $n-m$ must be even, which is impossible if $n$ is odd and $m$ is even.
\end{description}
Therefore, in each case $z$ is MBM whenever it is of the type $(1,1)$.
\end{proof}

\begin{corollary}[Root criterion for existence of rational curves] \label{cor:rationalcuexc}
Let $X$ be an irreducible holomorphic symplectic manifold, and let
\[
z\in H^{1,1}(X,\mathbb Z)
\]
be a primitive class with $q(z)<0$. Suppose that $z$ satisfies the root condition
of Theorem \ref{Theorem:thmMBMclass}. Then $X$ contains a rational curve.
\end{corollary}

\begin{proof}
By Theorem \ref{Theorem:thmMBMclass}, the class $z$ is MBM. Hence the
hyperplane $z^\perp$ is one of the MBM hyperplanes appearing in the
wall-and-chamber decomposition of the positive cone. Therefore the Kähler cone
of $X$ is strictly smaller than the positive cone, due to Theorem \ref{theorem:mbmposk}.

By the Huybrechts description of the Kähler cone \cite[corollary 3.3]{huybrechts2003kahler}, the Kähler cone is
cut out from the positive cone by the classes of rational curves. Hence, if the
Kähler cone is strictly smaller than the positive cone, $X$ contains a rational
curve.
\end{proof}

The following question was asked in \cite{amerik2021rational}. 
\begin{question}
    Let $\Lambda$ be a lattice of IHS type. Does there exist a $z\in \Lambda$ such that it is MBM whenever it is of the type $(1,1)$.
\end{question}
Theorem~\ref{Theorem:thmMBMclass} gives partial answer to the above question.
\begin{corollary}
    Let $\Lambda$ be a lattice of IHS manifold type. Fix the dimension $2n$. Suppose that the largest negative root of the Riemann--Roch polynomial is less than or equal to $-4$, and $n$ is odd. Then every primitive integral vector $z\in \Lambda$ with $q(z) = -2$ is MBM whenever it is of the type $(1,1)$.
\end{corollary}

\begin{proof}
    Observe that $r$ is a root of $RR_X(t)$ if and only if $\frac{r}{q(z)}$ is a root of $P_{\mathfrak{M}^o(\Lambda)}^z (x)$. Therefore, the smallest positive root of $P_{\mathfrak{M}^o(\Lambda)}^z (x)$ is greater than or equal to $2$, because $q(z) =-2$. $n$ is odd and $1\in (0,2)$, so $z$ is MBM whenever it is of the type $(1,1)$ due to Theorem \ref{Theorem:thmMBMclass}.
\end{proof}
Despite the immense effort devoted to this subject, only four deformation types of irreducible holomorphic symplectic manifolds are currently known. Beauville introduced two infinite series: IHS manifolds of the $K3^{[n]}$-type and of the generalized Kummer type \cite{beauville1983varietes}. In addition, there are two sporadic examples, in dimensions $6$ and $10$, known as the $\mathrm{OG}_6$ and $\mathrm{OG}_{10}$ deformation types, constructed by O'Grady \cite{o1997desingularized,o2000new}. The above corollary is not new for the known deformation types of irreducible holomorphic symplectic manifolds, but the proof is independent of what deformation type it is.
\begin{example}
    If $\Lambda$ is a lattice of type $K3^{[n]}$. Then by \cite[Example 1.2]{Jiang2020PositivityOR}, we have
    \[
     P_{\mathfrak{M}^o(\Lambda)}^z (x) = \frac{1}{n!}\prod_{k=2}^{n+1} \Bigl( \frac{q(z)x}{2}+k\Bigr).
    \]
    When $q(z)$ is equal to $-2$, the roots are 
    \[
    2 , \  3,\  \cdots \ ,\ n, \ n+1.
    \]
     If $n$ is odd, the above corollary implies $z$ is MBM whenever it is of type $(1,1)$. This type of result has already been proved in \cite[Theorem 1.2, Proposition 1.5]{10.1215/21562261-2081243}.
\end{example}
\begin{example}\label{EX:1}
    If $\Lambda$ is a lattice of $OG_{10}$ type. Then by \cite[Theorem 2]{RiosOrtiz2024RR}, we have 
    \[
    P_{\mathfrak{M}^o(\Lambda)}^z (x) = \frac{1}{n!}\prod_{k=2}^{6} \Bigl( \frac{q(z)x}{2}+k\Bigr).
    \]
    When $q(z)$ is equal to $-2$, the roots are 
    \[
    2, \ 3,\ 4,\ 5,\ 6.
    \]
     The above corollary implies $z$ is MBM whenever it is of type $(1,1)$. This type of result has already been proved in \cite[Proposition 1]{Mongardi2023ErratumMonodromy}.
    
\end{example}
We should not expect to prove the converse of the above theorem as stated because of the example below.
\begin{example}\label{EX:2}
    For $K3^{[2]}$ type, we have 
    \[
    P_{\mathfrak{M}^o(\Lambda)}^z (x) = \frac{1}{2} \Bigl( \frac{q(z)x}{2}+2\Bigr)\Bigl( \frac{q(z)x}{2}+3\Bigr).
    \] 
    The roots are 
    \[
    \frac{-4}{q(z)}, \ \frac{-6}{q(z)}.
    \]
    By \cite[Theorem~4.1]{amerik2019mbm}, an integral $(1,1)$ class $z$ is MBM then $q(z) = -2 \text{ or } -10$. In either case, our analysis gives nothing.
\end{example}

\begin{example}
We give an example of a poor IHS manifold of $K3^{[2]}$-type with positive Picard rank. This example also illustrates the sharpness of the open-interval condition in Theorem~\ref{Theorem:thmMBMclass}. Let
\[
\Lambda=U^{\oplus 3}\oplus E_8(-1)^{\oplus 2}\oplus\langle-2\rangle
\]
be the Beauville--Bogomolov lattice of an IHS manifold of $K3^{[2]}$-type, and let $\mathfrak M^0_\Lambda$ be a connected component of the moduli space of marked IHS manifolds with lattice $\Lambda$. Choose a standard basis $e,f$ of one copy of $U$, so that $q(e)=q(f)=0$ and $q(e,f)=1$, and set $z:=e-2f$. Then $z$ is primitive and $q(z)=-4$. Consider the locus
\[
Q_\Lambda\cap\mathbb P(z^\perp\otimes\mathbb C).
\]
Choose a very general point of this locus. By surjectivity of the period map on $\mathfrak M^0_\Lambda$, there exists a marked IHS manifold $(X,\varphi)\in \mathfrak M^0_\Lambda$ having this period. Since the point is very general in $Q_\Lambda\cap\mathbb P(z^\perp\otimes\mathbb C)$, every integral $(1,1)$-class on $X$ is proportional to $z$. Since $z$ is primitive, we obtain
\[
\Pic(X)=\NS(X)=\mathbb Zz.
\]
In particular, $\rho(X)=1$, and since $q(z)=-4<0$, the Picard lattice of $X$ is negative definite. Thus $X$ is elliptic. We claim that $X$ is poor. By \cite[Theorem~4.1]{amerik2019mbm}, a primitive integral MBM class on an IHS manifold of $K3^{[2]}$-type has Beauville--Bogomolov square $-2$ or $-10$. Since $q(z)=-4$, the class $z$ is not MBM. Since every rational $(1,1)$-class on $X$ is proportional to $z$, $X$ has no MBM classes. Indeed, if $X$ contained a rational curve $C$, then under the Beauville--Bogomolov identification its class would correspond to a nonzero rational $(1,1)$-class, hence to a rational multiple of $z$. Thus, after replacing $z$ by $-z$ if necessary, a positive rational multiple of $z$ would be represented by a rational curve. By \cite[Theorem~5.10]{amerik2015rational}, $z$ would then be MBM, a contradiction. Therefore $X$ contains no rational curves. Since $X$ is elliptic, Theorem~\ref{theorem:nocurvespoor} implies that $X$ contains no curves at all and is poor. There is an additional feature of this example. Let $L$ be the line bundle with $c_1(L)=z$. For $K3^{[2]}$-type, we have,
\[
P^z_{\mathfrak M^0_\Lambda}(x)=\frac12\left(\frac{q(z)x}{2}+2\right)\left(\frac{q(z)x}{2}+3\right).
\]
Since $q(z)=-4$, we obtain
\[
P^z_{\mathfrak M^0_\Lambda}(x)=\frac12(2-2x)(3-2x).
\]
Its positive roots are $1$ and $\frac32$. Thus $1$ is itself a root, rather than lying in one of the open intervals appearing in Theorem~\ref{Theorem:thmMBMclass}. Hence Theorem~\ref{Theorem:thmMBMclass} gives no conclusion that $z$ is MBM, consistently with the fact that $z$ is not MBM. Moreover,
\[
\chi(X,L)=P^z_{\mathfrak M^0_\Lambda}(1)=0.
\]
Since $X$ is poor and $q(c_1(L))=-4<0$, Proposition~\ref{prop:negative-square-cohom} gives $H^i(X,L)=0$ for every $i\neq2$. But Proposition~\ref{prop:negative-square-cohom} also gives $h^2(X,L)=\chi(X,L)=0$, and therefore
\[
H^i(X,L)=0\qquad\text{for every }i.
\]
Thus $X$ is a poor IHS manifold of $K3^{[2]}$-type with positive Picard rank which carries a nontrivial acyclic line bundle of negative Beauville--Bogomolov square.
\end{example}

For $K3$ surfaces, the  above Theorem together with results in \cite{vikash2026classificationpoormanifoldslow}, we can show that a negative curve is extremal if and only if $C^2 =-2$, which is a classical result. We will explore this connection in future work. 

\section{Classification of poor elliptic IHS manifolds}
\subsection{Poor elliptic IHS manifolds}
Let $X$ be a IHS manifold. Since $H^1(X,\mathbb{Z}) = 0$, the exponential sequence shows that
\[
    \operatorname{Pic}(X) \cong \operatorname{NS}(X)
    = H^{1,1}(X) \cap H^2(X,\mathbb{Z}).
\]
We are primarily interested in the lattice $(\operatorname{NS}(X), q_X)$, whose rank is the Picard number $\rho$.
\begin{proposition} \label{proposition:proppic0}
    Let $X$ be an IHS manifold with $\operatorname{Pic}(X) = 0$. Then $X$ is poor. In fact, $X$ has no curves at all.
\end{proposition}

\begin{proof}
        First, $X$ has no codimension-one subvarieties. Now suppose, for contradiction, that $X$ contains a curve $C \subset X$. The Beauville--Bogomolov--Fujiki form, together with Poincar\'e duality, induces a $\mathbb{Q}$-linear isomorphism between
        \[
        H^{1,1}(X,\mathbb{Q})
        \quad\text{and}\quad
        H^{2n-1,2n-1}(X,\mathbb{Q}).
        \]
        In particular, for any curve $C \subset X$, the cohomology class $[C]^\vee$ corresponds to a unique rational $(1,1)$-class and some nonzero multiple is an integral class, hence gives a line bundle, a contradiction.
\end{proof}
Let us now concentrate on IHS manifolds with nontrivial Picard group and fix some notation.
\begin{definition}
    Let $X$ be an IHS manifold with Picard number $\rho \neq 0$, and consider the lattice $(\operatorname{NS}(X), q)$. Then $X$ is called:
    \begin{enumerate}
        \item \emph{hyperbolic} if $(\operatorname{NS}(X), q)$ has signature $(1,0,\rho-1)$;
        \item \emph{parabolic} if $(\operatorname{NS}(X), q)$ has signature $(0,1,\rho-1)$;
        \item \emph{elliptic} if $(\operatorname{NS}(X), q)$ has signature $(0,0,\rho)$.
    \end{enumerate}
\end{definition}
A result of Huybrechts states that an IHS manifold is projective if and only if it is hyperbolic (see \cite{HuybrechtsErratum,huybrechts1997compact}). If $X$ is poor, then by Corollary \ref{corollary:concor} any line bundle $L$ with $q(c_1(L)) = 0$ is nef or its dual is nef.
\begin{conjecture}[SYZ Conjecture {\cite{soldatenkov2024abundance}}]
    Let $L$ be a nef line bundle on a hyperk\"ahler manifold $X$ with $q_X(c_1(L)) = 0$. Then there exists a Lagrangian fibration $\pi \colon X \to B$ and an ample line bundle $\mathcal{O}_B(1)$ on $B$ such that
    \[
        L^{\otimes k} \simeq \pi^*\bigl(\mathcal{O}_B(1)\bigr)
    \]
    for some positive integer $k$.
\end{conjecture}
\begin{remark}\label{remark:syz}
    If we assume the SYZ conjecture then from the above discussion, a manifold $X$ cannot be both poor and parabolic. Indeed, if $X$ were parabolic and poor, then any isotropic class with $q(c_1(L)) = 0$ would be nef, and the SYZ conjecture would produce a Lagrangian fibration $\pi \colon X \to B$. The results of \cite{soldatenkov2024abundance} imply that such a fibration has singular fibers, while \cite{hwang2009characteristic} shows that a general singular fiber is covered by rational curves. This contradicts the assumption that $X$ is poor.
\end{remark}
Therefore, it is natural to restrict ourselves to poor elliptic IHS manifolds. In this subsection, we prove some properties of poor elliptic IHS manifolds. 
\begin{definition}
    An IHS manifold $X$ is called elliptic if its Picard lattice is negative definite with respect to the BBF form.
\end{definition}
\begin{theorem}\label{theorem:nocurvespoor}
    Let $X$ be an elliptic IHS manifold. Then the following are equivalent:
    \begin{enumerate}
        \item $X$ is poor.
        \item $X$ contains no curves.
        \item $X$ contains no rational curves.
    \end{enumerate}
\end{theorem}

 \begin{proof}
        We first prove $(1) \Rightarrow (2)$. Assume that $X$ is poor, and suppose for a contradiction that $X$ contains a curve. Then for any curve $C \subset X$, there exists a class $\alpha \in H^{1,1}(X,\mathbb{R})$ associated to $C$ with 
        \[
            q(\alpha) < 0
        \]
        (see Remark \ref{remark:bbf-dual-curve-classes}). By Corollary~\ref{corollary:concor}, this implies that 
        \[
            \alpha \notin \overline{\mathcal{K}_X}.
        \]
        By \cite[Proposition~5.1]{huybrechts2003kahler}, the existence of such a class then guarantees that $X$ contains a rational curve, contradicting the assumption that $X$ is poor. The implication $(2) \Rightarrow (3)$ is immediate. Finally, we show $(3) \Rightarrow (1)$. Assume that $X$ contains no rational curves, and suppose for a contradiction that $X$ is not poor. By definition, this means that $X$ contains a codimension-one subvariety. Let $D \subset X$ be an irreducible component of such a divisor. In the elliptic case one has
        \[
            q\big(c_1(\mathcal{O}_X(D))\big) < 0.
        \]
        By \cite[Proposition~5.4]{huybrechts2003kahler}, the divisor $D$ is then uniruled, so it is covered by rational curves. This contradicts the assumption that $X$ contains no rational curves. Therefore $X$ has no divisors, and together with (3) this shows that $X$ is poor. Hence $(3) \Rightarrow (1)$.
    \end{proof}
As we saw in Proposition \ref{proposition:proppic0}, the same phenomenon occurs for elliptic IHS manifolds as well. However, the absence of codimension-one subvarieties is not the same as the absence of curves. In particular, an IHS manifold may contain rational curves while having no codimension-one subvarieties as we shall see in the next section. We now prove some results about elliptic IHS manifolds.
\begin{definition}
Let $\mathcal{F}$ be a coherent sheaf over an $n$-dimensional compact K\"ahler manifold $X$. We define the degree $\deg(\mathcal{F})$ as
\[
    \deg(\mathcal{F}) = \frac{\displaystyle\int_X c_1(\mathcal{F}) \wedge \omega^{n-1}}{\operatorname{vol}(X)},
\]
and the slope $\mu_{\omega}(\mathcal{F})$ as
\[
    \mu_{\omega}(\mathcal{F}) = \frac{1}{\operatorname{rank}(\mathcal{F})} \cdot \deg(\mathcal{F}).
\]
\end{definition}

\begin{definition}[Slope stability]
A torsion-free coherent sheaf $\mathcal{F}$ on a compact K\"ahler manifold $(X,\omega)$ is called slope stable (respectively, slope semistable) if for every nonzero coherent subsheaf $\mathcal{E} \subsetneq \mathcal{F}$ with $\operatorname{rank}(\mathcal{E}) < \operatorname{rank}(\mathcal{F})$ one has
\[
    \mu_{\omega}(\mathcal{E}) < \mu_{\omega}(\mathcal{F}) \quad (\text{respectively, } \mu_{\omega}(\mathcal{E}) \leq \mu_{\omega}(\mathcal{F})).
\]
\end{definition}

By the Kobayashi--Hitchin correspondence, proven by Donaldson and Uhlenbeck--Yau (see \cite{lubke1995kobayashi, uhlenbeck1986existence}), the tangent bundle $T_X$ is slope stable with respect to any K\"ahler class $\omega \in \mathcal{K}_X$. In particular, $\Omega_X^1 = T_X^\vee$ is also slope stable of slope $0$ with respect to any such $\omega$.

\begin{theorem}\label{theorem:thstb}
    Let $X$ be an elliptic IHS manifold of dimension $2n$. Then either
    \begin{enumerate}
        \item $X$ has an MBM class, or 
        \item $T_X$ (equivalently, $\Omega_X^1$) has no non-trivial torsion-free subsheaves of smaller rank.
    \end{enumerate}
\end{theorem}

 \begin{proof}
Assume that $X$ has no MBM classes and suppose, for contradiction, that there exists a nonzero torsion-free subsheaf
\[
0\neq \mathcal F \subsetneq \Omega_X^1
\]
of rank $r<2n$. Fix a K\"ahler class $\omega\in \mathcal K_X$. Since $\Omega_X^1$ is slope-stable of slope $0$ with respect to every K\"ahler class, we have
\[
\mu_\omega(\mathcal F)<\mu_\omega(\Omega_X^1)=0.
\]
Hence
\[
\mu_\omega(\mathcal F^*)=-\mu_\omega(\mathcal F)>0.
\]
Let
\[
L:=\det(\mathcal F^*),
\qquad
c:=c_1(L)\in H^{1,1}(X,\mathbb Z).
\]
Since $\operatorname{rk}(\mathcal F^*)=r$, we get
\[
\mu_\omega(L)=r\,\mu_\omega(\mathcal F^*)>0
\]
for every K\"ahler class $\omega\in \mathcal K_X$. Equivalently,
\[
\int_X c\wedge \omega^{2n-1}>0
\qquad
\text{for all }\omega\in \mathcal K_X.
\]
In particular, $c\neq 0$. Because $X$ has no MBM classes, Theorem~\ref{theorem:mbmposk} gives
\[
\mathcal K_X=\mathcal C_X,
\]
where $\mathcal C_X$ is the positive cone. By Fujiki's relation, there exists a positive constant $C_X$ such that for all $\alpha,\beta\in H^2(X,\mathbb R)$,
\[
\int_X \alpha\wedge \beta^{2n-1}
=
C_X\,n\, q(\beta)^{\,n-1} q(\alpha,\beta).
\]
Applying this with $\alpha=c$ and $\beta=\omega$, and using $\omega\in \mathcal C_X$ so that $q(\omega)>0$, we obtain
\(
q(c,\omega)>0
\text{ for all }\omega\in \mathcal C_X.
\)
Now $(H^{1,1}(X,\mathbb R),q)$ has signature $(1,b_2(X)-3)$, so the dual cone of the positive cone $\mathcal C_X$ is the closure $\overline{\mathcal C_X}$. Therefore the condition
\[
q(c,\omega)>0 \quad\text{for all }\omega\in \mathcal C_X
\]
implies that
\(
c\in \overline{\mathcal C_X}.
\)
In particular,
\(
q(c)\ge 0.
\)
On the other hand, $c\in H^{1,1}(X,\mathbb Z)\cap H^2(X,\mathbb Z)=\operatorname{NS}(X)$ is nonzero, and $X$ is elliptic. By definition, this means that the Beauville--Bogomolov form is negative definite on $\operatorname{NS}(X)$. Hence every nonzero class in $\operatorname{NS}(X)$ satisfies
\(
q(c)<0,
\)
contradicting $q(c)\ge 0$. This contradiction shows that no such $\mathcal F$ can exist. Therefore, if $X$ has no MBM classes, then $\Omega_X^1$ (equivalently, $T_X$) has no nonzero torsion-free subsheaf of rank $<2n$.
\end{proof}
\begin{corollary}\label{corollary:folfib}
Let $X$ be a poor manifold and an elliptic IHS manifold. Then $X$ admits no holomorphic foliations and no fibrations.
\end{corollary}

\begin{proof}
    Since $X$ is poor, $X$ does not admit MBM classes due to Proposition \ref{proposition:propcone} and Theorem \ref{theorem:mbmposk}. Then by Theorem \ref{theorem:thstb}, $TX$ has no torsion-free subsheaves; therefore, there are no holomorphic foliations and fibrations.
\end{proof}

\subsection{Families of poor elliptic IHS manifolds}
Let $X$ be an irreducible holomorphic symplectic manifold. $X$ admits a universal deformation space, which we denote by $\operatorname{Def}(X)$. Thus, there exists a proper flat family
\[
\pi : \mathcal{X} \longrightarrow \operatorname{Def}(X), \qquad \pi^{-1}(0) \cong X.
\]
\begin{definition}
Let $Y$ be a complex space and $n \ge 0$ an integer. A compact $n$–cycle on $Y$ is a finite formal sum
\[
  Z = \sum_{i=1}^r m_i [V_i],
\]
where each $V_i \subset Y$ is an irreducible compact analytic subset of pure dimension $n$ and each $m_i \in \mathbb{N}$ is a positive integer. Let,
\[
\operatorname{Supp}(Z) := \bigcup_i V_i.
\]
We denote by $\mathcal{C}_n(Y)$ the Barlet space of compact $n$-cycles on $Y$.
\end{definition}
For each integer $n \ge 0$, let $\mathcal{C}_n(\mathcal{X})$ denote the Barlet space of compact $n$–cycles in $\mathcal{X}$. The relative cycle space of $\pi$ is defined as
\[
\mathcal{C}_n(\mathcal{X}/\operatorname{Def}(X))
:= \bigl\{ [Z] \in \mathcal{C}_n(\mathcal{X}) \,\bigm|\,
\operatorname{Supp}(Z) \text{ is contained in a single fiber of } \pi \bigr\}.
\]
By Barlet’s theorem (see \cite{barlet2006espace} or Theorem~9.2.1 in \cite{magnusson2005lectures}), both
$\mathcal{C}_n(\mathcal{X})$ and $\mathcal{C}_n(\mathcal{X}/\operatorname{Def}(X))$ carry natural structures of reduced complex spaces and $\mathcal{C}_n(\mathcal{X}/\operatorname{Def}(X))$ is a closed complex subspace of $\mathcal{C}_n(\mathcal{X})$. Moreover, since $\mathcal{X}$ is countable at infinity, the cycle spaces $\mathcal{C}_n(\mathcal{X})$ and $\mathcal{C}_n(\mathcal{X}/\operatorname{Def}(X))$ are also countable at infinity and therefore have at most countably many irreducible components (see \cite[Corollary~11.1.4]{magnusson2005lectures}).
\begin{definition}
    A poor IHS manifold $X$ is called  truly poor IHS manifold, if $H^{1,1}(X,\mathbb{Q}) = 0$, or for every nonzero $ v \in H^{1,1}(X,\mathbb{Q})$, we have $q(v) < 0$.
\end{definition}
\begin{theorem}\label{thm:maain}
Let $X$ be an irreducible holomorphic symplectic manifold, and let
\[
\pi : \mathcal{X} \longrightarrow \operatorname{Def}(X)
\]
denote its universal deformation family. Define the following subsets of $\operatorname{Def}(X)$:
\begin{enumerate}
    \item 
    \[
    \operatorname{Def}(X)_{\mathrm{tp}}
    := \bigl\{\, t \in \operatorname{Def}(X) \,\bigm|\,
    X_t := \pi^{-1}(t) \text{ is truly poor} \,\bigr\};
    \]
    \item 
    \[
    \operatorname{Def}(X)_{\mathrm{c}}
    := \bigl\{\, t \in \operatorname{Def}(X) \,\bigm|\,
    X_t \text{ contains a curve} \,\bigr\};
    \]
    \item 
    \[
    \operatorname{Def}(X)_{\ge 0}
    := \bigl\{\, t \in \operatorname{Def}(X) \,\bigm|\,
    \exists\, 0 \neq L \in \operatorname{Pic}(X_t) \text{ such that } q\bigl(c_1(L)\bigr) \ge 0 \,\bigr\};
    \]
\end{enumerate}
Then
\[
\operatorname{Def}(X)_{\mathrm{tp}}
= \operatorname{Def}(X)\setminus
\bigl( \operatorname{Def}(X)_{\mathrm{c}} \cup \operatorname{Def}(X)_{\ge 0} \bigr).
\]
Moreover, $\operatorname{Def}(X)_{\mathrm{tp}}$ is the complement of a countable union of proper analytic subvarieties of $\operatorname{Def}(X)$.
\end{theorem}\label{thm:picn0poor}

\begin{proof}
    First, let us prove that $\operatorname{Def}(X)_c$ is a countable union of closed proper analytic subvarieties inside $\operatorname{Def}(X)$. Choose an exhaustion of $\operatorname{Def}(X)$ by polydiscs
    \[
    U_1 \Subset U_2 \Subset \cdots,\qquad \bigcup_{i\ge 1} U_i=\operatorname{Def}(X),
    \]
    where each $U_i$ is simply connected and relatively compact in $\operatorname{Def}(X)$. For each $i$, set $\mathcal{X}_{U_i}:=\pi^{-1}(U_i)$ and let
    \[
    \pi_i:\mathcal{X}_{U_i}\to U_i
    \]
    denote the restricted family.
    \begin{description}
        \item[Step 1] Consider the relative Barlet space $\mathcal{C}_1(\mathcal{X}_{U_i}/U_i)$ of $1$-cycles in the fibers of $\pi_i$, together with the natural map
        \[
        p_i:\mathcal{C}_1(\mathcal{X}_{U_i}/U_i)\longrightarrow U_i.
        \]
        Since each $U_i$ is simply connected and relatively compact in $\operatorname{Def}(X)$ Propositions 4.2 and Corollary 4.3 of \cite{bakker2022algebraic} imply that the restriction of $p_i$ to each irreducible component of $\mathcal{C}_1(\mathcal{X}_{U_i}/U_i)$ is proper. In particular, the image $p_i\bigl(\mathcal{C}_1(\mathcal{X}_{U_i}/U_i)\bigr)\subset U_i$ is a countable union of analytic subvarieties.
        \item[Step 2] We also have the map
        \begin{align*}
        p : \mathcal{C}_1(\mathcal{X}/\operatorname{Def}(X)) &\longrightarrow \operatorname{Def}(X), \\\
        [Z] &\longmapsto \pi(\operatorname{Supp}(Z)).
        \end{align*}
        Clearly, \(p\bigl(\mathcal{C}_1(\mathcal{X}/\operatorname{Def}(X))\bigr)=\operatorname{Def}(X)_c\) and
        \[
        p\bigl(\mathcal{C}_1(\mathcal{X}/\operatorname{Def}(X))\bigr) \cap U_i = p_i\bigl(\mathcal{C}_1(\mathcal{X}_{U_i}/U_i)\bigr).
        \]
        Since each irreducible component of \(p_i\bigl(\mathcal{C}_1(\mathcal{X}_{U_i}/U_i)\bigr)\) is closed, \(\operatorname{Def}(X)_c\) is a countable union of closed sets in \(\operatorname{Def}(X)\). Write
        \[
        \mathcal{C}_1(\mathcal{X}/\operatorname{Def}(X)) = \bigcup_m Z_m, \qquad 
        \mathcal{C}_1(\mathcal{X}_{U_i}/U_i) = \bigcup_m W_{m,i},
        \]
        where each \(Z_m\) is irreducible. From the above discussion, there are only countably many \(Z_m\). Set
        \[
        A_m := p(Z_m).
        \] 
        Then
        \[
        \operatorname{Def}(X)_c = \bigcup_m A_m.
        \]
        It remains to show that each \(A_m\) is an analytic subvariety of \(\operatorname{Def}(X)\). Define
        \[
        Z_m^{(i)} := Z_m \cap \mathcal{C}_1(\mathcal{X}_{U_i}/U_i).
        \] 
        Then \(Z_m^{(i)}\) is contained in a unique irreducible component \(W_{m',i}\). By Step~1, the restriction
        \[
        p_i|_{W_{m',i}} : W_{m',i} \longrightarrow U_i
        \] 
        is proper. Since \(Z_m^{(i)}\) is an analytic subvariety of \(W_{m',i}\), its image is an analytic subvariety of \(U_i\). Moreover,
        \[
        p_i\bigl(Z_m^{(i)}\bigr) = A_m \cap U_i.
        \]
        Because this holds for every \(i\), the sets \(A_m \cap U_i\) are analytic subvarieties on an open cover of \(\operatorname{Def}(X)\). Hence, each \(A_m\) is a closed analytic subvariety of \(\operatorname{Def}(X)\). Consequently, \(\operatorname{Def}(X)_c\) is a countable union of analytic subvarieties in \(\operatorname{Def}(X)\).
        \item[Step 3] Finally, $\operatorname{Def}(X)_c$ is proper: by \cite{campana1983densite}, the set
        \[
        \operatorname{Def}(X)_{\rho=0} := \{t \in \operatorname{Def}(X) \mid \text{the Picard rank of } X_t \text{ is } 0\}
        \]
        is dense in $\operatorname{Def}(X)$, and Proposition \ref{proposition:proppic0} implies that $\operatorname{Def}(X)_c$ is proper. 
    \end{description}
    Let us prove that $\operatorname{Def}(X)_{\ge 0}$ is a countable union of proper analytic subvarieties of $\operatorname{Def}(X)$. Fix a lattice isomorphism
    \(
    H^2(X,\mathbb{Z}) \cong \Lambda.
    \)
    For each $z \in \Lambda$, consider the subset
    \[
    \operatorname{Def}(X)_{z} := \bigl\{\, t \in \operatorname{Def}(X) \,\bigm|\, \exists\, L \in \operatorname{Pic}(X_t) \text{ with } c_1(L) = z \bigr\}.
    \]
    By \cite[1.16]{huybrechts1997compact}, for every nonzero $z$ the set $\operatorname{Def}(X)_{z}$ is empty or a hypersurface in $\operatorname{Def}(X)$. Hence
    \[
    \operatorname{Def}(X)_{\ge 0} = \bigcup_{\substack{z \in \Lambda \\ q(z) \ge 0}} \operatorname{Def}(X)_{z},
    \]
    so $\operatorname{Def}(X)_{\ge 0}$ is a countable union of proper analytic subvarieties of $\operatorname{Def}(X)$. The equality
    \[
    \operatorname{Def}(X)_{\mathrm{tp}}
    = \operatorname{Def}(X)\setminus
    \bigl( \operatorname{Def}(X)_{\mathrm{c}} \cup \operatorname{Def}(X)_{\ge 0} \bigr)
    \]
     follows from Theorem \ref{theorem:nocurvespoor}. Therefore, $\operatorname{Def}(X)_{\mathrm{tp}}$ is the complement of a countable union of proper analytic subvarieties in $\operatorname{Def}(X)$.
\end{proof}
The most natural question one could ask is: Are there any poor IHS manifold with nonzero Picard group? We know such $K3$ surfaces exist from \cite{vikash2026classificationpoormanifoldslow}. Let us try to answer that question in higher dimensions.
\begin{theorem}\label{thm:picneq0}
Let $\Lambda$ be an integral lattice of IHS manifold type, with rank greater than or equal to $6$. Then there exists an irreducible holomorphic symplectic manifold $X$ with $\operatorname{Pic}(X)\neq 0$, which is poor, together with an isomorphism
\[
\phi : H^2(X,\mathbb{Z}) \xrightarrow{\ \cong\ } \Lambda.
\]
\end{theorem}

\begin{proof}
    Let us give the proof in steps.
    \begin{description}
        \item[Step 1] Let us prove that there exists an infinite sequence of primitive vectors $v_n \in \Lambda$ such that $q(v_n) \to -\infty$, where $q$ is the quadratic form associated to $\Lambda$. Since $\Lambda$ is an integral lattice of IHS manifold type, there exists a primitive vector $v \in \Lambda$ with $q(v) < 0$. Extend $v$ to a basis $(v, e_1, \dots, e_k)$ of $\Lambda$ and set
        \[
        v_n := n v + e_1, \qquad n \in \mathbb{Z}.
        \]
        In this basis, the coordinates of $v_n$ are $(n,1,0,\dots,0)$, so $\gcd(n,1,0,\dots,0)=1$ and therefore each $v_n$ is primitive. Moreover,
        \[
        q(v_n) = q(nv + e_1)
           = n^2 q(v) + 2n\, q(v,e_1) + q(e_1).
        \]
        Since $q(v) < 0$ and $q(v_n)$ is a quadratic polynomial in $n$ with a negative leading coefficient, we have $q(v_n) \to -\infty$ as $|n|\to\infty$.
        \item[Step 2] Let $\mathfrak{M}^0_{\Lambda}$ be a connected component of the moduli space $\mathfrak{M}_{\Lambda}$ of marked irreducible holomorphic symplectic manifolds with Beauville-Bogomolov lattice $\Lambda$. By \cite[Corollary~5.2]{amerik2017morrison} there exists a constant $C>0$ such that
        \[
        |q(z)| < C
        \]
        for every primitive MBM class $z$ of type $(1,1)$ on any manifold in $\mathfrak{M}^0_{\Lambda}$ because rank of $\Lambda$ is greater than or equal to $6$. Using Step~1, we can choose a primitive class $z \in \Lambda$ with $q(z) < -C$. In particular, $z$ cannot be MBM for any irreducible holomorphic symplectic manifold in $\mathfrak{M}^0_{\Lambda}$.
        \item[Step 3] From Step~2 we obtain a class $z \in \Lambda$ such that $q(z) < -C$.
        By surjectivity of the period map on the connected component $\mathfrak{M}^0_{\Lambda}$, there exists a marked IHS manifold $(Y,\psi) \in \mathfrak{M}^0_{\Lambda}$ such that
        \[
        \mathcal{P}([(Y,\psi)]) \in
        Q_{\Lambda} \cap \mathbb{P}(z^{\perp} \otimes \mathbb{C}).
        \] 
        For this $Y$, the Hodge decomposition satisfies
        \[
        H^{1,1}(Y,\mathbb{R}):= H^{1,1}(Y) \cap H^2(Y,\mathbb{R})
        = \mathcal{P}([(Y,\psi)])^{\perp} \subset \Lambda_{\mathbb{R}},
        \]
        so $z \in H^{1,1}(Y) \cap H^2(Y,\mathbb{Z})$, i.e.\ $z$ is a $(1,1)$-class on $Y$. If $\Pic(Y)$ is generated by $z$, then by Step~2 the class $z$ is not MBM. Therefore, by \cite[Theorem~5.15]{amerik2015rational}, $Y$ contains no rational curves. By Theorem~\ref{theorem:nocurvespoor} this implies that $Y$ is poor, which completes the proof in this case. If instead $\Pic(Y)$ is not generated by $z$, we deform $Y$ inside its Kuranishi space $\operatorname{Def}(Y)$ to a point $t$ such that $z$ remains of type $(1,1)$ on $Y_t$ and $\Pic(Y_t)$ is generated by $z$. This is possible by \cite{huybrechts1997compact}. Applying the previous argument to $Y_t$ then shows that $Y_t$ does not contain any rational curves and is poor, as required.
    \end{description}
\end{proof}

The above assumption is not unnatural: it is conjectured that every IHS manifold $X$ satisfies $b_2(X)\ge 7$. We impose the hypothesis $b_2(X)\ge 6$ in order to apply results from \cite{amerik2017morrison}.
\subsection{Examples}
 We now give examples (following \cite{HassettTschinkel2010ExtremalRays}) of IHS manifolds containing rational curves but admitting no codimension-one subvarieties. In particular, even in the elliptic IHS manifold setting, the absence of codimension-one subvarieties is strictly weaker than being poor.
\begin{example}\label{exp:AVexample}
By \cite[Theorem~4.1]{amerik2019mbm}, there exists an IHS manifold $X$ of $K3^{[2]}$-type and an MBM class$ z \in H^{1,1}(X,\mathbb{Q})$ such that $q(z) = -\frac{5}{2}$ and the full MBM locus of $z$ is isomorphic to $\mathbb{P}^2$ (for the definition of the full MBM locus, see \cite[Definition~1.5]{amerik2021contraction}). Deform $X$ to a nearby fiber $ X_1 \in \operatorname{Def}(X,z) $ such that $ \operatorname{rank}\operatorname{Pic}(X_1) = 1, $ and let $ z' \in H^{1,1}(X_1,\mathbb{Q})$ be the corresponding class. Then $ q(z') = q(z) = -\frac{5}{2}.$ Moreover, $z'$ is MBM, since the MBM property is deformation invariant as long as the class remains of type $(1,1)$, by \cite[Corollary~5.13]{amerik2015rational}. Since $z'$ generates $\operatorname{Pic}(X_1)$ over $\mathbb{Q}$, after replacing $z'$ by $-z'$ if necessary, \cite[Theorem~5.10]{amerik2015rational} implies that $z'$ is represented by an extremal rational curve. Hence, by \cite[Theorem~4.1]{amerik2019mbm}, its full MBM locus $ Z_{z'} \subset X_1 $ is isomorphic to $\mathbb{P}^2$. We claim that $X_1$ contains no codimension-one subvarieties. Suppose, to the contrary, that $ D \subset X_1$ is an irreducible divisor. Since $\operatorname{Pic}(X_1)$ has rank one and is generated over $\mathbb{Q}$ by $z'$, there exists $\lambda \in \mathbb{Q} \setminus \{0\}$ such that $ [D] = \lambda z'. $ Consequently, $ q(D) = \lambda^2 q(z') < 0. $ By \cite[Proposition~4.7]{boucksom2004divisorial}, every prime divisor of negative Beauville--Bogomolov square on an IHS manifold is uniruled. Thus $D$ is covered by rational curves. Let $C \subset D$ be a rational curve through a general point of $D$. Under the Beauville--Bogomolov identification $ H_2(X_1,\mathbb{Q}) \simeq H^2(X_1,\mathbb{Q}),$ the class $[C]$ is a rational $(1,1)$-class. Since $ H^{1,1}(X_1,\mathbb{Q}) = \mathbb{Q} z',$ we have $ [C] \in \mathbb{Q} z'. $ Therefore $C$ is contained in the full MBM locus $Z_{z'}$. Since such curves pass through a general point of $D$, we obtain $ D \subset Z_{z'}.$ But $ Z_{z'} \simeq \mathbb{P}^2, $ whereas $D$ is a divisor in the fourfold $X_1$, and hence $ \dim D = 3.$ This is impossible. Therefore $X_1$ contains no codimension-one subvarieties. On the other hand, $Z_{z'} \simeq \mathbb{P}^2$ is nonempty and is swept out by rational curves. Hence $X_1$ contains rational curves. We have therefore constructed an IHS manifold $X_1$ of $K3^{[2]}$-type which contains rational curves but no codimension-one subvarieties. 
\end{example}
In abitary dimension one could do the following. As pointed out to us by B.~Bakker, the construction of Bakker--Lehn \cite[Example~9.3]{BL22}, based on their deformation
theory of symplectic contractions \cite[Theorem~4.1 and Propositions~4.5 and~5.8]{BL21},
yields, for every $n\geq2$, a Picard-rank-one IHS manifold of $K3^{[n]}$-type containing a Lagrangian $\mathbb P^n$ but no divisors, by deforming a small contraction of $\mathbb P^n$
along the Hodge locus of its line class.
\begin{question}
    Let $X$ be an \emph{irreducible holomorphic symplectic} manifold with no rational curves. Is it true that $X$ contains no curves at all ?
\end{question}

\section*{Acknowledgments} I am grateful to my advisors, John Lesieutre, Yuriy Zarhin and Tatiana Bandman, for useful, stimulating discussions, and very helpful comments. I would like to thank Ekaterina Amerik for several helpful discussions and valuable comments. I am grateful to B.~Bakker for helpful comments. I would like to thank several graduate students at Penn State (Andy B. Day, Eugene Henninger-Voss, Neelarnab Raha, Satwata Hans, Xingkai Wang and Louis Diaz) for helpful discussions.

\end{document}